\documentclass[12pt]{amsart}

\usepackage{a4wide}
\usepackage{color}
\usepackage{amsrefs}
\usepackage{hyperref}

\newtheorem{thm}{Theorem}
\newtheorem{lem}{Lemma}
\newtheorem{cor}{Corollary}
\newtheorem{prop}{Proposition}
\newtheorem{exa}{Example}

\title[Hermitian and skew hermitian forms over local rings]{Hermitian and skew hermitian forms over local rings}

\author{James Cruickshank}
\address{School of Mathematics, Statistics and Applied Mathematics, National University of Ireland Galway, Ireland}

\email{james.cruickshank@nuigalway.ie}

\author{Rachel Quinlan}
\address{School of Mathematics, Statistics and Applied Mathematics, National University of Ireland Galway, Ireland}

\email{rachel.quinlan@nuigalway.ie}

\author{Fernando Szechtman}
\address{Department of Mathematics and Statistics, University of Regina, Canada}

\email{fernando.szechtman@gmail.com}


\begin{document}

\maketitle
\begin{abstract}
    We investigate the structure of possibly degenerate \( \varepsilon \)-hermitian forms 
    over local rings. We prove classification theorems in the cases where the ring is complete and either 
    the form is nondegenerate or the ring is a discrete valuation ring. In the latter case we 
    describe a complete set of invariants for such forms based on a generalisation of the classical notion 
    of the radical of the form.
\end{abstract}

\section{Introduction}

We are concerned with the classification problem of possibly degenerate 
hermitian  or skew hermitian bilinear forms over a local ring in which 2 is a unit.
Symmetric and skew symmetric forms are included as a special case, as we allow the underlying involution to be trivial.

In foundational papers of Durfee (\cite{MR0011073}) and O'Meara (\cite{MR0067163})
the integral theory of quadratic forms over local fields is developed.
This is equivalent to the theory of symmetric bilinear forms over a complete local principal ideal domain with finite residue field. 
Jacobowitz (\cite{MR0150128}) extended this integral theory to the case of 
hermitian forms over local fields equipped with an involution. 

More recently, Levchuk and Starikova (\cite{LO}) have 
proved the existence and uniqueness of normal forms 
for symmetric matrices over local principal ideal domains
under certain assumptions on the unit group of the ring.

In a wider setting 
Bayer-Fluckiger and Fainsilber (\cite{MR1374198}) have considered the 
general problem of equivalence of hermitian or skew hermitian forms 
over arbitrary rings and are able to prove some quite general reduction theorems
for this problem.
A broad study of sesquilinear forms and their connection to hermitian 
forms was carried out by Bayer-Fluckiger, First and Modovan (\cite{MR3245846})
as well as by Bayer-Fluckiger and Moldovan (\cite{MR3124208}).

For general information on quadratic and hermitian forms over rings we refer
the reader to the textbook of Knus (\cite{MR1096299}).

\subsection{Outline of the paper}
Given the extensive existing literature on this topic and the fact that we aim to provide a unified and self contained exposition, some overlap with 
previously known results is inevitable. In this section we summarise the results of the paper, emphasising the contributions that we have made.

Throughout, \( A \) is a commutative local ring with maximal ideal
\( \mathfrak r \) and \( * \) is an involution of \( A \), that is 
to say an
automorphism of order at most two.
Clearly, \( * \) induces an involution on the residue field \( A/\mathfrak r \).
If this induced involution is trivial, we say that \( * \) is 
{\em ramified}, otherwise it is {\em unramified} -
this distinction will play a decisive role in the sequel. 
Thus broadly speaking there are four types of forms to consider, ramified 
hermitian, ramified skew hermitian, unramified hermitian and unramified skew hermitian. 
However, the last two can be merged into one case 
since they differ only up to multiplication by a unit.

In Section \ref{sec:smallnondegen} we present some elementary results on nondegenerate
forms of rank one or two in various cases. These essentially form the building 
blocks for our later classification theorems. In light of this we include all the details 
so as to make the paper self contained.

In Section \ref{sec:completelocalrings} we introduce completeness and use it to derive the existence of basis vectors with certain specified properties. A key 
new result here is Lemma \ref{lem:completenessimplessymplecticpairs} which 
guarantees the existence of a symplectic pair under appropriate conditions. We also give an example to show that this lemma can fail in the absence of completeness.

In Section \ref{sec:nondegenforms} we use the results of the previous sections to analyse nondegenerate forms of any rank over arbitrary local rings. 
In this case there is an associated vector space and form over the residue field and in 
the case that the ring is also complete we show that equivalence of forms over the local ring reduces to equivalence of these associated forms over the residue field. 
In particular we demonstrate
the existence of a symplectic basis for a nondegenerate skew hermitian form over a complete local ring with ramified involution - see Proposition \ref{prop:sympbasis}. 

The remaining sections of the paper deal with possibly degenerate forms. 
In this case it is not true without further assumptions that such forms have a nice 
- in a sense precisely defined in Section \ref{sec:DVRs} -
decomposition. However in the case of a discrete valuation ring, that is to say a local principal ideal domain, we are able to prove
that any form has such a nice decomposition - this is the content of Theorem \ref{thm:dvrcanonicaldecomp}. 

Following this we introduce a generalisation of the classical notion of the radical of a form. In our case we have one generalised radical module for each nonnegative integer.
Using these generalised radicals we are able to show that the equivalence problem for a pair of forms over a complete discrete valuation ring can be reduced to 
the equivalence of a series of pairs of nondegenerate forms over the associated residue field - see Theorem \ref{thm:henselforformsoverdvr}. This is a broad generalisation of Theorem D of \cite{MR2555069} which dealt only with the special case of symmetric matrices.

Under some additional hypotheses, these spaces essentially turn out to give rise to a complete set of invariants for hermitian and skew hermitian forms over a complete discrete valuation ring. This can be found in Theorem \ref{lindo}.

Also under these hypotheses, we are able to prove, in Theorems 
\ref{lindo2} and \ref{unra}, the existence and uniqueness of normal forms in all possible cases. As a consequence of these normal forms we also prove that the congruence class of a hermitian or skew hermitian matrix, over a complete local principal ideal domain satisfying the aforementioned additional 
hypotheses, depends only on the invariant factors of the
matrix - this is stated in Theorem \ref{lindo3}.
Results analogous to this last theorem have been obtained elsewhere,
for polynomial rings over algebraically closed fields see Theorem 4.5 of \cite{MR1960118},
and for skew symmetric matrices over principal ideal domains see Exercise 4 of Chapter XIV of \cite{MR783636}.

\subsection{Terminology and notation}
\label{subsec:terminology}

We shall (sometimes) use the shorthand \(
\overline a \) for the residue class \( a + \mathfrak r\).
Moreover, we extend this shorthand to matrices in the obvious way.
We write \(U(A)\) for the
multiplicative group of
units of \(A\) and recall the fundamental fact that, since \(A\)
is local, \( \mathfrak r = A -U(A)\).

Let \(R\) be the subring of \(A\) consisting of elements
fixed by \(*\). Observe that \( R\) is also a local ring,
with maximal ideal \(\mathfrak m = R \cap \mathfrak r\).
Let \( S = \{ a \in A: a^* = -a \}\).
Since \( 2 \) is assumed to be 
a unit of \(A\) we have \( A = R \oplus S\).
If \( S \subset \mathfrak r\) we say that \( *\) is
{\em ramified}. Otherwise we say that it is {\em unramified}.
Since \(\mathfrak r\) is \(*\)-invariant, \(*\) induces an
involution on the residue field \(A/\mathfrak r\), which we
also denote by \(*\).
As indicated in the introduction, \( *\) is ramified if and only if the induced
involution on the residue field is trivial.

We write \( M_m(A) \) for the set of \( m \times m \) matrices over \( A \). The 
group of invertible \( m \times m \) matrices is denoted by \( GL_m(A) \).

Let \(V\) be a free \(A\)-module of rank \(m>0\) and
let \(h:V \times V \rightarrow A\) be an
{\em \( \varepsilon \)-hermitian} form, where \( \varepsilon = \pm 1 \).
That is to say \( h \) is bi-additive, \( A \)-linear 
in the second variable and \( h(v,u) = \varepsilon h(u,v)^*\)
for all \( u,v \in V\).

The {\em Gram matrix} of a list of vectors \(v_1,\cdots, v_k\)
with respect to a form \(h\) is the \( k \times k\) matrix
whose \( (i,j)\)-entry is \(h(v_i,v_j)\).
The form \(h\) is said to be {\em nondegenerate} if the Gram matrix of
any basis of \( V \) belongs to \(GL_m(A)\).

A matrix \(C \in M_m(A)\) is said to be {\em \(\varepsilon\)-hermitian}
if \(C'^* = \varepsilon C\) where \(C'\) denotes the transpose
of \(A\).
Of course \( h\) is
\(\varepsilon\)-hermitian if and only if  \( G\)
is \(\varepsilon\)-hermitian, where \(G\) is the Gram
matrix of any basis. On the other hand, given an \(\varepsilon\)-hermitian
\(C \in M_m(A)\), we can construct an \(\varepsilon\)-hermitian
form \(h_C\) on \( V = A^m\), by \(h_C (u,v) = u'^*Cv\).

Matrices \(C,D \in M_m(A)\) are said to be {\em \(*\)-congruent} if
there exists \(X \in GL_m(A)\) such that \(D = X'^*CX\).
Given an \( \varepsilon \)-hermitian form \(h\), then
the Gram matrices of any two bases are \(*\)-congruent.
On the other hand we say that forms \(h_1\) on \( V_1 \),
respectively   \( h_2\) on \(V_2\) are
{\em equivalent} if there is an \( A \)-isomorphisn
\( \varphi:V_1 \rightarrow V_2 \) such that
\(h_1(u,v) = h_2(\varphi(u),\varphi(v))\). Now, \(h_C\) and \(h_D\) as above
are
equivalent if and only \( C\) and \(D\) are \(*\)-congruent.

\section{Small nondegenerate submodules}
\label{sec:smallnondegen}

In contrast to the field case,
a linearly independent set is not necessarily a subset of
a basis - this is apparent even in the rank one case.
In the presence of the form \(h\), however, we have some
sufficient conditions.

Given an \(A\)-submodule \(W\) of \(V\), let
\( W^\perp = \{ v \in V: h(w,v) = 0 \}\). Then
\(W^\perp\) is an \(A\)-submodule of \(V\).

\begin{lem}
    \label{lem:unitlengthimpliesbasisvector}
    Let \(u \in V\) and suppose that \(h(u,u) \in U(A)\).
    Then \(V =
    Au \oplus Au^\perp\). Moreover, both \(Au\) and
    \(Au^\perp\) are free \(A\)-modules.
\end{lem}

\begin{proof}
    Let \(v_1,\cdots,v_m\) be a basis of \(V\).
    So \(u = \sum c_i v_i\). Now \( h(u,u) = \sum
    c_i h(u,v_i) \in U(A)\). Since
    \(A\) is local this implies that \(c_j \in U(A)\)
    for some \(j\). Without loss of generality assume that \(
    c_1 \in U(A)\). Therefore \(u,v_2,\dots,v_m\)
    is a basis of \(V\). Now, for \(i = 2,\cdots m\),
    let \(w_i = v_i - h(u,v_i)h(u,u)^{-1}u\). Clearly
    \(u,w_2,\cdots,w_m\) is a basis of \(V\) and
    \(w_2,\cdots,w_m\) is a basis of \(Au^\perp\).
\end{proof}

In the case that \(*\) is ramified and \(h\) is skew hermitian, it
is immediate that \( h(u,u) \in \mathfrak r\) for all \(u \in V\).
So we cannot have any nondegenerate rank one submodules in that case.
However we have the following lemma concerning rank two submodules.

\begin{lem}
    \label{lem:skewunitproduct}
    Suppose that \( *\) is ramified and that
    \(h\) is skew hermitian. If \(u,v \in V\)
    satisfy
    \( h(u,v) \in U(A) \), then
    \( V = (Au+Av) \oplus (Au+Av)^\perp\). Moreover
    both summands in this decomposition are free
    \(A\)-modules.
\end{lem}

\begin{proof}
    Let \(v_1,\cdots,v_m\) be a basis of \(V\). So \(u = \sum_i a_iv_i\).
    Suppose that \(a_i \in \mathfrak r\) for
    all \(i\). Then clearly \( h(u,v) \in \mathfrak r\), contradicting
    our assumption. So without loss of generality we may assume that \(a_1 \in U(A)\).
    It follows that \(u,v_2,\cdots,v_n\) is a basis of \(V\).
    Write \( v = b_1u + \sum_{j=2}^m b_jv_j\). If \(b_2,\cdots,b_n \in
    \mathfrak r\), then since \( h(u,u) \in \mathfrak r\)
    it would follow that \( h(u,v) \in \mathfrak r\), contradicting our
    hypothesis. Without loss of generality we may assume that \(b_2 \in U(A)\) and it
    follows that \(u,v,v_3,\cdots,v_m\) is a basis of \(V\).
    Since the Gram matrix of \(u,v\) is invertible in this case
    it follows that, for \(i = 3 ,\cdots,m\), there exist
    (unique) \(a_i,b_i \in A\) such that
    \(h(u,a_iu+b_iv) = h(u,v_i)\) and \( h(v,a_iu+b_iv) = h(v,v_i)\).
    Let \(w_i = v_i - (a_iu+b_iv)\). Then \(u,v,w_3,\cdots,w_m\)
    is a basis of \(V\) and \(w_2,\cdots,w_3\) is a basis of
    \( (Au+Av)^\perp\) as required.
\end{proof}

Now we show, excepting the ramified skew hermitian case, that if
\(h\) has any unit value, then there is some element \(w \in V\)
such that \(h(w,w)\) is a unit.

\begin{lem}
    \label{lem:unitproductimpliesunitlength}
    Let $u,v\in V$.
    Suppose that \(h(u,v) \in U(A)\) and that either
    \( h\) is hermitian or that \( *\) is unramified.
    Then there is some \( w \in V\) such that \(h(w,w)
    \in U(A)\).
\end{lem}

\begin{proof}
    Replacing \( v \) by \(h(u,v)^{-1}v\) if necessary,
    we can assume without loss of generality that
    \(h(u,v) = 1\). Now suppose that \(h\) is hermitian.
    Then
    \[2 = h(u,v) +h(v,u) = h(u+v,u+v)-h(u,u)-h(v,v). \]
    Now since \(2 \in U(A)\) and \(A\) is local, we
    conclude that at least one of \(h(u,u)\),
    \(h(v,v)\) or \(h(u+v,u+v)\) belongs to \(U(A)\).
    Finally in the case where \(h\) is skew hermitian
    and \( *\) is unramified, choose \(b \in U(A)
    \cap S\). Now observe that \( bh\) is
    hermitian and that \( bh(u,v) \in U(A) \)
    if and only if \(h(u,v) \in U(A)\).
\end{proof}

In summary the results of this section show that if \(h\) has any
unit value then it is possible to break off a nondegenerate
submodule of rank one except in the ramified skew hermitian case.
In the latter case, no nondegenerate rank one submodules can possibly exist,
but it is possible to break off a nondegenerate
submodule of rank two.

\section{Complete local rings}
\label{sec:completelocalrings}

Having given conditions sufficient to ensure the existence
of nondegenerate rank one and two submodules,
in this section we introduce some natural conditions
on the ring that will guarantee the existence of unit length
basis vectors, or in the ramified skew hermitian case, the existence of a
symplectic pair. Recall that a symplectic pair is an
ordered pair of vectors \( (u,v) \in V^2\) satisfying
\( h(u,u) = h(v,v) = 0\) and \( h(u,v) = -h(v,u) = 1\).
We also recall that a symplectic basis is a basis \(
v_1,w_1,v_2,w_2,\cdots,v_l,w_l\) such that each pair \( (v_i,w_i)\) is
a symplectic pair and such that \( h(v_i,v_j) = h(w_i,w_j) = h(v_i,w_j) = h(w_j,v_i) = 0\) for \(i \neq j\).

Observe that \(h(au,au) = a^*ah(u,u)\), so it
is natural to investigate the image of the so-called
norm map, \( N: a \mapsto a^*a\).
To this end, we follow Cohen (\cite{C}) and say that $A$ is {\em complete} if \(\bigcap_{i=1}^\infty \mathfrak r^i = 0\)
and $A$ is metrically complete with respect to its \(\mathfrak r\)-adic
metric. We observe that
in the present context, where we seek solutions to certain quadratic equations,
it is natural to restrict our attention to complete rings.

\begin{lem}
    \label{lem:henselwithinvolution}
    Suppose that \( A \) is complete. Suppose that
    \(b \in U(R)\) and that
    \( a^*a \equiv b \mod \mathfrak r\) for some $a\in A$. Then there is some
    \(c \in A\) such that \( c^*c = b\).
\end{lem}

\begin{proof}
    First we observe that since \( A\) is complete,
    \(R\) is also complete. Since $2\in R$,  it follows from 
    Hensel's Lemma (see \cite{C}, Theorem 4) that
    the squaring map from \( 1 + \mathfrak m \) to \( 1 + \mathfrak m\)
    is a surjection.
    Now \( a \in U(A)\) since \( b \in U(A)\), therefore \( b(a^*a)^{-1}
    \in 1 + \mathfrak m\). Hence there is some
    \( \delta \in \mathfrak m \) such that \( b(a^*a)^{-1}
    = (1+\delta)^2\). So \(b = (a(1+\delta))^*a(1+\delta)\).
\end{proof}

\begin{cor}
    \label{cor:henselinV}
    Suppose that \(A\) is complete and that $u\in V$, \( b \in U(A)\) 
    and \( b^* = \varepsilon b\).
    If \( h(u,u) \equiv b \mod \mathfrak r \) then there is some \( w \in
    Au\) such that \( h(w,w) = b\).
\end{cor}

\begin{proof}
    Note that \( h(u,u)^{-1}b \in 1 +\mathfrak m \).
    By Lemma \ref{lem:henselwithinvolution} there is some \( c \in A\)
    such that \( c^*c = h(u,u)^{-1}b\). Let \( w = cu \).
\end{proof}

Now we turn to the nondegenerate two dimensional submodules in
the case where \( *\) is ramified and \(h\) is skew hermitian.

\begin{lem}
    \label{lem:completenessimplessymplecticpairs}
    Suppose that \(A\) is complete and that
    \( * \) is ramified. Suppose also that
    \(h\) is skew hermitian and that \(h(u,v) \in
    U(A)\). Then there is some symplectic pair \( (u',v') \in V^2\) such that
    \(Au'+Av' = Au+Av \).
\end{lem}

\begin{proof}
    By replacing \(v\) by \(h(u,v)^{-1}v\) we can assume that
    \(h(u,v) = 1\). Now observe
    that for \( b \in A\), \(h(u+bv,u+bv) = 0\)
    if and only if
\begin{equation}\label{eq:quadone} b^*bh(v,v) +(b-b^*) + h(u,u) = 0\end{equation}
    Since both \(h(v,v) \) and \(h(u,u) \) belong to \(S\)
    we can apply Lemma \ref{lem:quadraticequationhasroot} (below) to conclude that Equation
    (\ref{eq:quadone}) has a solution in \( \mathfrak r \). So assume that \(b \in \mathfrak r\) satisfies
    (\ref{eq:quadone}) and let \(u' = u+bv\). Clearly, \(Au+Av = Au'+Av\).
    Moreover, \(h(u',v) = 1+b^*h(v,v) \in U(A)\). Now, let
    \(v' = h(u',v)^{-1}(v +\frac12 h(v,v)(h(u',v)^*)^{-1} u')\).
    A straightforward calculation, using \(h(u',u') =0\),
    demonstrates that \( (u',v') \) is
    a symplectic pair. Moreover, since \(h(u',v)\) is a unit,
    it is clear that \(Au'+Av' = Au'+Av = Au+Av\).
\end{proof}

To complete the proof of Lemma \ref{lem:completenessimplessymplecticpairs}
we need the following.

\begin{lem}
    \label{lem:quadraticequationhasroot}
    Suppose that \(A\) is complete and that \( *\)
    is ramified.
    Given \(\alpha, \beta, \gamma \in A\) satisfying \(\alpha,\gamma \in S\)
    and \(\beta \in U(A)\),
    there is some \( t \in A\gamma\) such that
    \[ \alpha t^*t + \beta t-t^*\beta^* +\gamma = 0.\]
\end{lem}

\begin{proof}
We define sequences \((\beta_i)\) and \((\gamma_i)\) as follows. Let
\(\gamma_1 = \gamma\) and \(\beta_1 = \beta\).
Given \(\beta_i\) and \(\gamma_i\), let
\[ \gamma_{i+1} = -\frac{\alpha\gamma_i^2}{4\beta_i^*\beta_i},\]
\[ \beta_{i+1} = \beta_i+\frac{\alpha\gamma_i}{2\beta_i^*}.\]

Observe that \(\beta_{i+1}\) is a unit since \(\beta_{i+1} \equiv \beta_i \mod \mathfrak r\).
Also \( \gamma_{i+1} \) is skew hermitian, since \(\alpha\) and \( \gamma_i\) (inductively)
are both skew hermitian.
Now we check that \( -\frac12 \sum_{i=1}^\infty \gamma_i/\beta_i \) is the required solution.
Define \(f_i(t) = \alpha t^*t +\beta_it-t^*\beta_i^* +\gamma_i\).
Using the fact that \(\gamma_i\) is skew hermitian, an easy calculation  shows that
\(f_i(t-\gamma_i/2\beta_i) = f_{i+1}(t)\). Therefore \(
f_{k+1}(0) = f_{k}(-\gamma_k /2\beta_k) = f_{k-1}
(-\gamma_k/2\beta_k-\gamma_{k-1}/2\beta_{k-1}) = \cdots = f_1(-
\frac12\sum_{i=1}^k\gamma_i/\beta_i)\).
Now \(f_{k+1}(0) = \gamma_k\). Since \(\gamma_1 \in \mathfrak r\),
it is clear that \(\gamma_k \in \mathfrak r^k\) for all \(k\).
Therefore \(f_1( -\frac12 \sum_{i=1}^k \gamma_i/\beta_i)
\in \mathfrak r^k\) as required.
\end{proof}

The following example shows that if \( A \) is not complete it may be
possible to find a nondegenerate rank two skew hermitian submodule
that does not have a symplectic basis.
\begin{exa}
    \label{exa:Rachel}
    {\rm
        Fix an odd prime $p$ such that $p+1$ is not the sum of two squares
        (e.g. $p=5$) and let \( A \) be the extension of \( R=\mathbb Z_{(p)} \) obtained
        by adjoining a square root of~\( p \). So \( A = R\oplus R\sqrt{p} \)
        is a local principal ideal domain with maximal ideal \( A\sqrt{p} \).
        Let \( * \) be the involution of \( A \) that fixes elements of \( R \)
        and maps \( \sqrt{p} \) to \( -\sqrt{p} \) and consider the
        nondegenerate
        skew hermitian form \( h_M \) where \( M = \begin{pmatrix}
        \sqrt{p} & 1 \\ -1 & \sqrt{p} \end{pmatrix}\). We claim that
        there is no isotropic basis vector of \( A^2 \) .
        Suppose that \( \begin{pmatrix}1\\ a +b\sqrt{p}\end{pmatrix} \)
        was such an vector (a similar argument
        applies to the case \( \begin{pmatrix}a+b\sqrt{p} \\1 \end{pmatrix}
        \)). A straightforward calculation shows that
        we must have \( p b^2-2b-1=a^2 \). Writing \( b = \frac{c}{d} \)
        for integers \( c \) and \(d \),
        we see that \( pc^2-2cd-d^2 \)
        must be the square of an integer. But \(pc^2-2cd -d^2 = (p+1)c^2 - (d+c)^2 \),
        so \((p+1)c^2\) is the sum of two squares. By assumption, $p+1$ is not the sum of two squares, so neither is
        \( (p+1)c^2 \) for any integer \( c \).
    }
\end{exa}

\section{Nondegenerate forms}
\label{sec:nondegenforms}

We set $V(-1)=0$ and
\[ V(i) = \{ v \in V: h(V,v) \subset \mathfrak r^i\},\quad i\geq 0. \]
Observe that \( V(i)\) is an \(A\)-submodule of \(V\), \(V(i+1)
\subset V(i) \) and \(\mathfrak r V(i-1)
\subset V(i) \) for all \(i \geq 0\). We note in passing that, if \( \bigcap_{i\geq 0} \mathfrak r^i=0 \) then \( \bigcap_{i\geq 0} V(i) =\{v\in V\,|\, h(v,V)=0\}\) is the {\em radical} of \(
h\). Now, since \(\mathfrak r V(i) \subset V(i+1)\), we see that
$$
W(i)=V(i)/(\mathfrak r V(i-1)+V(i+1))
$$
inherits a \(A/\mathfrak r\)-vector space structure. The form \(h\) induces
a map  $$h_i : W(i)\times W(i) \rightarrow \mathfrak r^i/\mathfrak r^{i+1}.$$
In the special case \(i=0\), we see that \(h_0\) is an \(\varepsilon\)-hermitian
form \(V/V(1) \times V/V(1) \rightarrow A/\mathfrak r\). Observe that $h$ is nondegenerate
    if and only if $h_0$ is nondegenerate.

\begin{lem}\label{lem:ri} Suppose that \(h\) is nondegenerate. Then $V(i)=\mathfrak r^i V$ for all $i\geq 0$.
\end{lem}

\begin{proof} The inclusion $\mathfrak r^i V\subset V(i)$ is clear. For the reverse inclusion, let $v\in V(i)$ and let $\{v_1,\dots,v_m\}$ be a basis of $V$, so that $v=a_1v_1+\cdots+a_mv_m$ for some $a_i\in A$. Since $h$ is nondegenerate, given any $1\leq j\leq m$, there
is $u\in V$ such that $h(u,v_j)=1$ and $h(u,v_k)=0$ for all $k\neq j$. It follows that $h(u,v)=a_j$, whence $a_j\in \mathfrak r^i$ and a fortiori $v\in \mathfrak r^i V$.
\end{proof}

\begin{thm}
    \label{thm:canonicalformnondegenerate}
    Suppose that \(h\) is nondegenerate.
    \begin{enumerate}
        \item If \( * \) is unramified or if \(h\)
            is hermitian then \(V\) has
            a basis whose Gram matrix is diagonal.
        \item If \( *\) is ramified and \( h \) is
            skew hermitian then \(V\) has a basis whose
            Gram matrix is the direct sum of
            \( m/2 \) matrices, each of
            which is an invertible \( 2 \times 2 \)
            skew hermitian matrix.
    \end{enumerate}
\end{thm}

\begin{proof}
    Since \(h\) is nondegenerate
    the Gram matrix of a basis is invertible.
    Since \(A\) is local
    some entry of this Gram matrix must be a unit.
    Therefore there exist vectors \(u,v\)
    such that \(h(u,v) \in U(A)\).
    The results of Section \ref{sec:smallnondegen}
    show that in the ramified skew hermitian case there is a nondegenerate
    rank two submodule of \(V\) and that in all other cases there is a
    nondegenerate rank one submodule. Moreover, if \(U\) is this
    rank one or two submodule then by
    Lemma \ref{lem:unitlengthimpliesbasisvector}
    or Lemma \ref{lem:skewunitproduct}, \(V = U \oplus U^{\perp}\)
    and \(U^\perp\) is a free submodule of \(V\).
    Now, it is clear that \(h|_{U^\perp}\) is also nondegenerate.
    The required conclusions follow by induction on the rank of
    \( V \).
\end{proof}

\begin{thm}
    \label{thm:henselforforms}
    Suppose that \(A\) is complete and let
    \(h\) and \(h'\) be nondegenerate \(\varepsilon\)-hermitian
    forms on \(V\). Then \( h\) and \(h'\) are equivalent if and
    only if \( h_0\) and \(h'_0\) are equivalent. In particular,
    if \( *\) is ramified and \(h\) and \(h'\) are skew
    hermitian, then $h$ and $h'$ are equivalent.
    \end{thm}
\addtocounter{thm}{-1}
\begin{thm}[Matrix version]
    Suppose that \(A\) is complete and let
    \( C , D \in GL_m(A) \). Then \( C \) and \( D \) are
    \( * \)-congruent if and only \( \overline C \) and
    \( \overline D \) are \( * \)-congruent over \( A/\mathfrak r \). In particular,
    if \( *\) is ramified and \(C\) and \(D\) are skew
    hermitian, then $C$ and $D$ are \( * \)-congruent.
\end{thm}
\begin{proof}
    The only if direction is obvious.
        For the other direction 
        suppose that \( \overline{C} \) and \( \overline{D} \) are 
        \( * \)-congruent. So there is 
        some matrix \( X \in M_m(A) \) such that \( \overline{X} 
        \in GL_m(A/\mathfrak r)\) and \( X'^*CX \equiv
        D \mod \mathfrak r\). Since \( A \) is local, \( X \in GL_m(A) \)
        and \( C \) is \( * \)-congruent to \( X'^*CX \).
        So replacing \( C \) by 
        \( X'^*CX \) we can, without loss of generality, reduce to the 
        case that \( C \equiv D \mod \mathfrak r \). Equivalently, we may assume 
        without loss of generality that 
        \( h(u,v) \equiv h'(u,v) \mod \mathfrak r \) for all \( u,v \in V \).

    Now we deal with the
    case where \( *\) is ramified and \(h\) is skew
    hermitian. In this case, Lemma \ref{lem:completenessimplessymplecticpairs}
    implies that $V$ has symplectic bases relative to $h$ and $h'$, whence they are equivalent.

    In all other cases, Lemmas \ref{lem:unitlengthimpliesbasisvector} and \ref{lem:unitproductimpliesunitlength} ensure
    the existence of a basis \(v_1,\cdots,v_m\) of $V$
    whose corresponding
    Gram matrix relative to $h$, say \( M \), is diagonal with only units on the diagonal.
    By Corollary \ref{cor:henselinV} there is some \(w_1
    \in Av_1\) such that \(h'(w_1,w_1) = h(v_1,v_1)\).
    Moreover, it is clear that, for \(j \geq 2\),
    \( h'(w_1,v_j)
    \equiv h'(v_1,v_j) \equiv h(v_1,v_j) \equiv 0
    \mod \mathfrak r\).
    For \(j \geq 2\), let
    \( w_j = v_j - h'(w_1,v_j)h'(w_1,w_1)^{-1}w_1\).
    Now \(w_1,\cdots,w_m\) is a basis of \(V\) whose
    Gram matrix with respect to \(h'\) is of the
    form
    \[ \begin{pmatrix} M_{11} & 0 \\ 0 & N\end{pmatrix}. \]
    Moreover \( N \equiv \text{diag}(M_{22},\cdots,M_{mm})
    \mod \mathfrak r\). Inductively, we may assume that there is
    some \( X \in GL_{m-1}(A)\) such that \(
    X'^*NX = \text{diag}(M_{22},\cdots,M_{mm})\). It follows
    that there is a basis of \(V\) whose Gram matrix
    with respect to \(h'\) is equal to \( M \).
    Note that Corollary \ref{cor:henselinV} provides the base case 
    of the induction.
\end{proof}

We list some  noteworthy corollaries and special cases of
Theorem \ref{thm:henselforforms}.
\begin{prop}
    \label{prop:sympbasis}
    If \( A \) is complete, \( * \) is ramified and \( h\) is
    skew hermitian and nondegenerate, then \( V \) has a symplectic basis.
    \qed
\end{prop}

We have a canonical
imbedding \( R/\mathfrak m\hookrightarrow A/\mathfrak r\) and we will view \( R/\mathfrak m\) as a subfield of~\(A/\mathfrak r\) by means
of this imbedding. If \( *\) is ramified then \( R/\mathfrak m = A/\mathfrak r\).
Suppose \( *\) is unramified. Then
\(R/\mathfrak m \) is the fixed field of an automorphism
of \(A/\mathfrak r\) of order 2. If \(A/\mathfrak r\) is quadratically closed (in the sense that it has no extensions of degree 2),
then by the Diller-Dress
theorem (see \cite{L}, p. 235), \( R/\mathfrak m \) is a Euclidean field (this is an ordered field wherein every nonnegative element is a square).

\begin{prop}\label{prop:C}
    Suppose that \(A\) is complete and that
    \( A/\mathfrak r \) is quadratically closed.
    If \( * \) is unramified and \( h \) is
    a nondegenerate hermitian form then
    \( V \) has a basis whose Gram matrix is diagonal and
    such that all the diagonal entries are \( \pm 1 \).
    Moreover, given any two such bases, the signatures of the
    corresponding Gram matrices are the same.
\end{prop}

\begin{proof}
    By Theorem \ref{thm:henselforforms}
    it suffices to prove this result in the case that
    \( A \) is a field (i.e. \( \mathfrak r = 0 \)). By our above remarks,
    $A$ is the quadratic extension of a Euclidean field. The result in this case
    goes exactly as in a classical case when $A=\mathbb C$.
     \end{proof}

Similarly we have

\begin{prop}
    Suppose that \(A\) is complete and that
    \( A/\mathfrak r \) is quadratically closed.
    If \( * \) is unramified and \( h \) is a nondegenerate
    skew hermitian form then,
    given \( b \in U(A)\cap S \),
    \( V \) has a basis whose Gram matrix is diagonal and
    such that all the diagonal entries are \( \pm b \).
    Moreover, given any two such bases, the
    number of occurrences of \( b \) are the same in each of the
    corresponding Gram matrices.
\end{prop}
\begin{proof}
    Apply the previous proposition to the hermitian form \( b^{-1}h \).
\end{proof}

\begin{prop}
    Suppose that \(A\) is complete and that
    \( A/\mathfrak r \) is quadratically closed.
    If \( * \) is ramified and \( h \) is a nondegenerate
    hermitian form then
    \( V \) has an orthonormal basis.
\end{prop}
\begin{proof}
    This follows by combining Theorem \ref{thm:henselforforms}
    with the classical result that every nondegenerate
    symmetric form over a quadratically closed field
    of characteristic not 2 has an orthonormal basis.
\end{proof}

\section{Discrete Valuation Rings}
\label{sec:DVRs}

For the remainder of the paper \(A\) will be
a discrete valuation ring, that is, a local principal
ideal domain. 
A {\em uniformiser} of \( A \) is a generator of the maximal
ideal.

\begin{lem}
    \label{lem:nontrivialimpliesskewgenerator}
    If \(*\) is a nontrivial involution of \(A\) then
    \(A\) has a uniformiser \( y \) such that \( y^* = -y\).
\end{lem}

\begin{proof}
    Suppose that \( \mathfrak r = Az \). Now \( z = \frac12(z+z^*)
    +\frac12(z-z^*)\) and since \(A\) is local, at least one of
    \(\frac12(z+z^*)\) or \(\frac12(z-z^*) \) must lie in \(\mathfrak r
    -\mathfrak r^2\). So we can certainly choose some \(w\) such that
    \(w^* = \pm w\) and \(\mathfrak r = Aw\). Suppose that \(w^* = w\).
    Since \(*\) is nontrivial there is some nonzero \(u \in A\) such that
    \(u^* = -u\). Now \(u = bw^k\) for some \(b \in U(A)\) and since
    \(w^* = w\), we conclude that \(b^* = -b\). Now \(y = bw\)
    is the required generator of \(\mathfrak r\).
\end{proof}

Following the previous lemma we fix \( y\) so that \(Ay = \mathfrak r\)
and so that \(y^* = -y\)
in the case that \(*\) is a nontrivial involution. We agree that $y^\infty=0$.

Following O'Meara \cite{MR0067163}
we say that an \( \varepsilon \)-hermitian
matrix \( M \in M_{ m } (A) \) has an {\em O'Meara decomposition} if
$$M=\bigoplus_{i=0}^{s}a_i M_i,$$
where each \( a_i \in A \) and  every \(M_i \) is an invertible
hermitian or skew hermitian matrix.

Our next result is an extension of Theorem \ref{thm:canonicalformnondegenerate} to arbitrary forms, possibly degenerate,
and establishes the existence of an O'Meara decomposition for the Gram matrices of such forms.

\begin{thm}
    \label{thm:dvrcanonicaldecomp}
        \begin{enumerate}
        \item If \( * \) is unramified then \( V \)
            has a basis whose Gram matrix is diagonal.
        \item If \( * \) is ramified then \( V \) has a
            basis whose Gram matrix is the direct sum of a diagonal
            matrix (possibly of size zero) and a number (possibly zero)
            of \( 2\times 2 \) blocks each of the form
            \( y^dB \) where \( B \) is an invertible skew
            hermitian \( 2 \times 2 \) matrix.

    \end{enumerate}
\end{thm}

\begin{proof}
    We prove this by induction on the rank of \( V \).
    If \( h \) is identically zero then the theorem is true.
    Now suppose that \( h \) is not identically zero. Let \(
    d = \min \{ j: h(V,V) \subset \mathfrak r^j \}\). Since \( A \) is a
    domain there is a unique form \( h' \) on \( V \) such that
    \( h = y^dh' \). Clearly \( h' \) is \( \varepsilon \)-hermitian
    if \( y^* = y \) and is \( (-1)^d\varepsilon \)-hermitian
    if \( y^* = -y \). Moreover, by construction, there exist
    \( u,v \in V \) such that \( h'(u,v) \in U(A) \).

    If \( h' \) is skew hermitian and \( * \)
    is ramified then, by Lemma \ref{lem:skewunitproduct}
    there is some nondegenerate rank two submodule \( U \) such that 
    \( V = U\perp U^\perp \). 
    If \( h' \) is hermitian or if \( * \) is unramified then,
    by Lemma \ref{lem:unitlengthimpliesbasisvector} 
    and Lemma \ref{lem:unitproductimpliesunitlength} there 
    is some nondegenerate rank one submodule \( U \) such that 
    \( V = U\perp U^\perp \). 

    In either case \( U \) has an O'Meara decomposition
    of the required form
    with respect to \( h' \) and by induction
    \( U^\perp \) has an O'Meara decomposition
    of the required form
    with respect to
    \( h' \). Together these yield the required
    decomposition for \( h \).
\end{proof}

Next we consider the structure of the space \( W(i)\) in more detail.
By Theorem \ref{thm:dvrcanonicaldecomp} we have
\begin{equation}
\label{dec}
V=U_0\perp U_1\perp U_2\perp\cdots\cdots\perp U_\infty
\end{equation}
where $U_i=0$ for all but finitely many $i$ and, for $U_i\neq 0$, the Gram matrix of a basis, say~$B_i$, of $U_i$ is equal to $y^i M_i$ with $M_i$ invertible. Thus, if $B$ is the union of all $B_i$ then $B$ is a basis of $V$ with Gram matrix
\begin{equation}
\label{dec2}
M=M_0\oplus y^1 M_1\oplus y^2M_2\oplus\cdots\cdots \oplus y^\infty M_\infty,
\end{equation}
where almost all summands have size 0.

\begin{lem}\label{prop:spanningset} For every nonnegative integer $i$, we have
$$
V(i)=\mathfrak r^i U_0\perp \mathfrak r^{i-1} U_1\perp \mathfrak r^{i-2} U_2\perp\cdots\perp U_i\perp  U_{i+1}\perp U_{i+2}\perp\cdots\perp U_\infty.
$$
\end{lem}

\begin{proof} The operation of passing from $V$ to $V(i)$ is compatible with orthogonal decompositions, so the result follows
immediately from Lemma \ref{lem:ri} and the decomposition (\ref{dec}).
\end{proof}

\begin{cor}  For every nonnegative integer $i$, we have $\dim_{A/\mathfrak r} W(i)=\mathrm{rank}_{A}\, U_i$.
\end{cor}

Making use of the uniformiser \( y \)
for \( A \), we may alter the above map $W(i)\times W(i)\to \mathfrak r^i/\mathfrak r^{i+1}$ and
define an \( A/\mathfrak r \)-valued
form \( h_i: W(i) \times W(i) \rightarrow A/\mathfrak r \) by
$$ h_i(u+\mathfrak r V(i-1)+V(i+1),v+\mathfrak r V(i-1)+V(i+1)) = y^{-i}h(u,v)+\mathfrak r,\quad u,v\in V(i).$$
Note that \( h_i \) is \( \varepsilon \)-hermitian if
\( * \) is trivial and is \( \varepsilon(-1)^i \)-hermitian
if \( * \) is nontrivial.

\begin{cor}\label{cor:matrix}  For every nonnegative integer $i$, the form $h_i$ is nondegenerate. In fact,
the Gram matrix of the basis $\overline{B_i}$ of $W(i)$ relative to $h_i$ is $\overline{M_i}$, which is invertible
or has size 0.
\end{cor}

It is clear that if \( h \) and \( h' \) are equivalent
forms on \( V \), then \( h_i \) and \( h'_i \) are
equivalent over \( A/\mathfrak r \)
for \( i \geq 0 \).
The converse is, of course, not necessarily true
(see Example \ref{exa:Rachel}).

\begin{thm}
    \label{thm:henselforformsoverdvr}
    Suppose that \( A \) is complete
    and
    that \( h \) and \( h' \) are \( \varepsilon \)-hermitian
    forms on \( V \).
    If,
    for each nonnegative integer \( i \),
    the forms \( h_i \) and \( h'_i \) are equivalent
    over \( A/\mathfrak r \) then \( h \) and \( h' \)
    are equivalent.
\end{thm}

\begin{proof}  This is an immediate consequence of Theorem \ref{thm:henselforforms} (in its matrix version), Corollary~\ref{cor:matrix}
and the decomposition (\ref{dec2}).
\end{proof}

We next obtain the following classification theorem for \( * \)-congruence
classes of \( \varepsilon \)-hermitian matrices

\begin{thm}
    \label{thm:matricesovercompleteDVR}
    Assume that \( A \) is complete.
    \begin{enumerate}
        \item Suppose that \( M \in M_m(A)\) is \( \varepsilon \)-hermitian.
            Then \( M \) is \( * \)-congruent to a matrix of the form
            \( \bigoplus_{i=0}^\infty y^i M_i \),
            where for each
            \( M_i \) is either of size zero or is an invertible
            matrix. Moreover, for $0\leq i<\infty$,
            \begin{enumerate}
                \item if \( * \) is unramified then \( M_i \)
                    is a diagonal matrix.
                \item if \( * \) is ramified and \( (y^i)^* = \varepsilon y^i \)
                    then \( M_i \) is diagonal. Whereas if \( * \) is ramified and \( (y^i)^* = -\varepsilon y^i \)
                    then \( M_i \) is a direct sum of
                    copies of \( \begin{pmatrix} 0 & 1 \\ -1 & 0 \end{pmatrix} \).
            \end{enumerate}
        \item Given \( \varepsilon \)-hermitian matrices
            \( M = \bigoplus_{i=0}^\infty y^i M_i \) and
            \( N = \bigoplus_{i=0}^\infty y^i N_i \) in $M_m(A)$ such that each
            \( M_i\) (resp. $N_i$) is either of size zero or is an invertible
            matrix, then
            \begin{enumerate}
                \item if \( * \) is unramified,
                    \( M \) and \( N \) are \( * \)-congruent over \( A \) if and only if
                    for each \( 0 \leq i<\infty \),
                    \( \overline{ M_i}\) is \( * \)-congruent to
                    \( \overline{ N_i} \) over \( A/\mathfrak r \).
                \item if \( * \) is ramified,
                    \( M \) and \( N \) are \( * \)-congruent over \( A \) if and only if
                    for each \( 0\leq i<\infty \) such that  \( (y^i)^* = \varepsilon y^i \),
                    \( \overline{M_i}\) is \( * \)-congruent to
                    \( \overline{N_i} \) over \( A/\mathfrak r \),
                    and for each \(0\leq i<\infty \) such that  \( (y^i)^* = -\varepsilon y^i \),
                    the size of \( M_i \) is equal to the size of~\( N_i \).

            \end{enumerate}
    \end{enumerate}
\end{thm}

\begin{proof} This follows immediately from Theorem \ref{thm:henselforforms}, Theorem \ref{thm:dvrcanonicaldecomp} and
Lemma \ref{lem:completenessimplessymplecticpairs}.
\end{proof}

Thus, for matrices over a complete
discrete valuation ring with residue field of characteristic not 2, the \( * \)-congruence problem essentially reduces to the
\( * \)-congruence problem over its residue field.

\begin{exa} (cf. Example 1.1 of \cite{LO}) {\rm Given any field $F$, consider the complete, local, nonprincipal domain $A=F[[X,Y]]$ and the $*$-hermitian (with trivial $*$) matrix
$$
M=\left(
    \begin{array}{cc}
      0 & X \\
      X & Y \\
    \end{array}
  \right).
$$
We claim that $M$ has no O'Meara decomposition. Suppose, to the contrary, that it does. Since $M$ is
not invertible, then $M$ must be $*$-congruent to a diagonal matrix $D=\mathrm{diag}(a,b)$.
Thus $D=T'MT$ for some $T\in\mathrm{GL}_2(A)$. Taking determinants yields $ab=-t^2X^2$, where $t\in U(A)$. Since $(X)$ is a prime ideal,
it follows that $a\in (X)$ or $b\in (X)$. Suppose, without loss of generality, that $a=cX$ for some $c\in A$. Cancelling, we obtain
$cb=-t^2X$. If $b\notin (X)$ then a repetition of the preceding argument yields that $b\in U(A)$, against the fact that
the hermitian form $h=h_M$ takes no unit values. We are thus forced to conclude that $b=dX$ for some $d\in A$, which implies
that $c,d\in U(A)$. Now, since $M$ is congruent to $D$, there is $u\in V=A^2$ such that $h(u,u)=cX$. Let $e_1,e_2$ be the canonical
basis of $V$. Then $u=f e_1+g e_2$ for some $f,g\in A$, whence $2fgX+g^2Y=cX$, that is, $g^2Y=(c-2fg)X$. We infer that $g\in (X)$
and a fortiori $c-2fg\in U(A)$. Thus $(c-2fg)X\in (Y)$ but neither $(c-2fg)$ nor $X$ are in $(Y)$, which contradicts the fact that $(Y)$ is a prime ideal.
}
\end{exa}

The following proposition is not needed for the sequel. However, we
include it to shed light on the structure of the rings under
consideration.

\begin{prop} Suppose \(*\) is a nontrivial involution of \(A\).
Then \(R\) is a discrete valuation ring and $A=R[z]=R\oplus Rz$, where $z^*=-z$ and $z^2=x\in R$. Moreover,
$x\in U(R)$ if $*$ is unramified and $x\in\mathfrak m$ otherwise. In either case, $x$ is not a square in $R$.
\end{prop}

\begin{proof} By Lemma
    \ref{lem:nontrivialimpliesskewgenerator} we can
    choose \(y \in S \cap (\mathfrak r - \mathfrak r^2)\).

    Suppose first that \(S \subset \mathfrak r\).
    Since \( \mathfrak m = \mathfrak r \cap R\),
    we know that if
    \( b \in \mathfrak m\) then
    \( b = cy\) for some \( c \in A\). Clearly, since \( b \in R\) and
    \( y \in S\), we must have \( c \in S \subset \mathfrak r\).
    Therefore \( c \in Ay\) and we conclude that \( b \in R y^2\).
    So in this case \( \mathfrak m = Ry^2\).  Since $2\in U(A)$,
    we have $A=R\oplus S$. If $s\in S$ then $sy\in \mathfrak r \cap R=\mathfrak m$,
    whence $sy=ry^2$ for some $r\in R$, so $s=ry$ and a fortiori $S=Ry$.
    Moreover, in this case, let \(x = y^2\) and observe that \( x\)
    cannot be a square of any element of \(R\).

    On the other hand, if \( S \not\subset \mathfrak r\)
    then we may choose \( w \in U(A)\) such that \(w^* = -w\).
    Then \(z = wy \in R \cap (\mathfrak r - \mathfrak r^2)\)
    and it is clear that \( \mathfrak m = Rz\) in this case.
    If $t\in S$ then $tw^{-1}\in R$, so $t\in Rw$, which gives $S=Rw$.
    Now let \(x = w^2 \in R\) and once again, one readily
    checks that $x$ is not a square in $R$.

    In both cases $R$ is a local domain with principal maximal ideal, so every ideal is principal.
    \end{proof}

\section{Invariants and normal forms}

We are finally in a position to show that, after imposing additional hypotheses, the sequence $d_i=\dim_{A/\mathfrak r} W(i)$, $i\geq 0$,
is a complete set of invariants for equivalence
classes of \( \varepsilon \)-hermitian forms.

Let $B$ stand for the fixed field of the involution that $*$ induces on $A/\mathfrak r$. Then $B=A/\mathfrak r$ if $*$ is ramified
and $B=R/\mathfrak m$ (viewed as a subfield of $A/\mathfrak r$) if $*$ is unramified. In either case, we have a norm map $\mathcal N:A/\mathfrak r\to B$ given by $a+\mathfrak r\to aa^*+\mathfrak r$. One of the aforementioned hypotheses is that $\mathcal N$ be surjective.

\begin{thm}\label{lindo}
    Suppose that \( A \) is complete and $\mathcal N$ is surjective.
    Let \( V \) and \( V' \) be free \( A \)-modules, both of
    rank \( m \), and let
    \( h:V\times V\to A\) and \(h:V'\times V'\to A \) be
    \( \varepsilon \)-hermitian forms.
    Then \( h \) and \( h' \) are equivalent
    if and only if $d_i=d'_i$ for all $i\geq 0$.
\end{thm}

\begin{proof} This follows from Corollary \ref{cor:matrix} and Theorem
    \ref{thm:henselforformsoverdvr} using the fact that $\mathcal N$ is surjective.
\end{proof}

We may use Theorem \ref{lindo} to obtain normal forms for \( \varepsilon \)-hermitian matrices.

\begin{thm}\label{lindo2}
    Suppose that \( A \) is complete and $\mathcal N$ is surjective.
    \begin{enumerate}
    \item If $*$ is nontrivial, let $M\in M_m(A)$ be skew hermitian (resp. hermitian).
    Then $M$ is $*$-congruent to one and only one matrix of
        the form $\bigoplus_{i=0}^\infty y^i M_i$, where every $M_i$ of size $>0$ is equal to
        the direct sum of copies of \( \begin{pmatrix} 0 & 1 \\ -1 & 0 \end{pmatrix} \) if $i$ is even (resp. odd) and
        equal to the identity matrix if $i$ is odd (resp. even).
    \item If $*$ is trivial, let $M\in M_m(A)$ be skew symmetric (resp. symmetric).
    Then $M$ is $*$-congruent to one and only one matrix of
        the form $\bigoplus_{i=0}^\infty y^i M_i$, where every $M_i$ of size $>0$ is equal to 
        the direct sum of copies of \( \begin{pmatrix} 0 & 1 \\ -1 & 0 \end{pmatrix} \) (resp.
        equal to the identity matrix).
     \end{enumerate}
\end{thm}

Recall that two matrices $M,N\in M_m(A)$ are said to be equivalent if $PMQ=N$ for some $P,Q\in \mathrm{GL}_m(A)$.
Assuming that $A$ is complete and $\mathcal N$ is surjective, we may use Corollary \ref{cor:matrix} and Theorem \ref{thm:henselforformsoverdvr} to see that the problem of $*$-congruence of matrices reduces to the
problem of equivalence of matrices, whose answer is well-known. Since every invertible matrix is equivalent to the identity matrix, we have

\begin{thm}\label{lindo3}
    Suppose \( A \) is complete and $\mathcal N$ is surjective. Then two \( \varepsilon \)-hermitian matrices $M,N\in M_m(A)$
   are $*$-congruent if and only if they have the same invariant factors.
\end{thm}

What are the possible invariant factors of an \( \varepsilon \)-hermitian matrix?
By above this amounts to asking what are the sequences
that arise as \( (d_i)_{i\geq 0} \) for some \( \varepsilon \)-hermitian
form.
Let us call such sequences
{\em \( \varepsilon \)-realisable}. The answer is an immediate consequence of~Theorem \ref{lindo2}.

\begin{prop}
    Suppose that \( A \) is complete and $\mathcal N$ is surjective.
    Let \( (d_i) \) be a sequence of nonnegative integers.
    \begin{enumerate}
        \item If \( * \) is trivial and \( \varepsilon = 1 \)
            then \( (d_i) \) is \( \varepsilon \)-realisable.
        \item If \( * \) is trivial and \( \varepsilon = -1 \)
            then \( (d_i) \) is \( \varepsilon \)-realisable if and only if
            \( d_i \) is even for all \( i \).
        \item If \( * \) nontrivial, then \( (d_i) \)
            is \( \varepsilon \)-realisable if and only if
            \( d_i\) is even
            for \( i \) such that  \( (-1)^i = -\varepsilon \).
    \end{enumerate}
\end{prop}

The analogue of Theorem \ref{lindo2} in the case when $*$ is unramified and $A/\mathfrak r$ is assumed to be quadratically closed
is an immediate consequence of Proposition \ref{prop:C} and
Theorem \ref{thm:matricesovercompleteDVR}.

\begin{thm}\label{unra} Suppose that \( A \) is complete with quadratically closed residue field.
    Assume that \( * \) is unramified and fix $b\in S\cap U(A)$. Let $M\in M_m(A)$ be hermitian (resp. skew hermitian).
    Then $M$ is $*$-congruent to one and only one matrix of
        the form $\bigoplus_{i=0}^\infty y^i M_i$, where every $M_i$ of size $>0$ is equal to a diagonal
        matrix $\mathrm{diag}(1,\dots,1,-1,\dots,-1)$ (resp. $\mathrm{diag}(b,\dots,b,-b,\dots,-b)$).
\end{thm}

\end{document}